\documentclass[%
 aip,
 amsmath,amssymb,
 reprint,%
]{revtex4-1}

\usepackage{graphicx}
\usepackage{dcolumn}
\usepackage{bm}

\usepackage[utf8]{inputenc}
\usepackage[T1]{fontenc}
\usepackage{mathptmx}
\usepackage{amsmath}
\usepackage{amsthm}
\usepackage{amssymb}
\newtheorem{theorem}{Theorem}
\newtheorem{lemma}[theorem]{Lemma}

\begin{document}

\preprint{AIP/12https://www.overleaf.com/project/607b700b5b6aed13ffadacf13-QED}

\title[Kuramoto model on a sphere] {The Kuramoto model on a sphere: Explaining its low-dimensional dynamics with group theory and hyperbolic geometry}

\author{Max Lipton}
\email{ml2437@cornell.edu}
 \affiliation{Department of Mathematics, Cornell University, Ithaca, NY 14853}
\author{Renato Mirollo}%
 \email{mirollo@bc.edu}
\affiliation{Department of Mathematics, Boston College, Chestnut Hill, MA 02467}%

\author{Steven H. Strogatz}
 \email{shs7@cornell.edu}
\affiliation{Department of Mathematics, Cornell University, Ithaca, NY 14853}%

\date{\today}

\begin{abstract}
We study a system of $N$ interacting particles moving on the unit sphere in $d$-dimensional space. The particles are self-propelled and coupled all to all, and their motion is heavily overdamped. For $d=2$, the system reduces to the classic Kuramoto model of coupled oscillators; for $d=3$, it has been proposed to describe the orientation dynamics of swarms of drones or other entities moving about in three-dimensional space. Here we use group theory to explain the recent discovery that the model shows low-dimensional dynamics for all $N \ge 3$, and to clarify why it admits the analog of the Ott-Antonsen ansatz in the continuum limit $N \rightarrow \infty$. The underlying reason is that the system is intimately connected to the natural hyperbolic geometry on the unit ball $B^d$. In this geometry, the isometries form a Lie group consisting of higher-dimensional generalizations of the Möbius transformations used in complex analysis. Once these connections are realized, the reduced dynamics and the generalized Ott-Antonsen ansatz  follow immediately. This framework also reveals the seamless connection between the finite and infinite-$N$ cases. Finally, we show that special forms of coupling yield gradient dynamics with respect to the hyperbolic metric, and use that fact to obtain global stability results about convergence to the synchronized state.
\end{abstract}

\maketitle

\begin{quotation}
Exactly solvable models have long played a central role in nonlinear dynamics, from Newton’s work on the gravitational two-body problem to breakthroughs in understanding solitons in the 1970s. Often, the solvability of a model reflects an underlying symmetry or other special structure in its governing equations. In this paper we discuss a many-body system of current interest, known as the Kuramoto model on a sphere, whose unexpectedly low-dimensional dynamics call out for explanation. The model consists of $N$ identical overdamped particles moving on a $(d-1)$-dimensional sphere in $d$-dimensional Euclidean space, yielding a state space of dimension $N (d-1)$. Yet despite the presence of damping, the model exhibits enormously many constants of motion. Here we show that its trajectories are confined to invariant manifolds of dimension $d(d+1)/2$ for all $N \ge 3$ and trace the origin of this low-dimensional behavior to an underlying group-theoretic structure in the system. Specifically, the Kuramoto model on a sphere turns out to be the flow induced by the action of the group of Möbius transformations on the $d$-dimensional ball, and its invariant manifolds are the associated group orbits. For certain forms of coupling, the model acquires further structure (hyperbolic gradient dynamics) that force almost all solutions to converge to the perfectly synchronized state.  
\end{quotation}  

\section{Introduction} In 1975, Kuramoto introduced a model for a large  population of coupled oscillators with randomly distributed natural frequencies.~\cite{Kuramoto1975self} Kuramoto’s  model displayed many remarkable features: It was exactly solvable (at least in some sense), despite being nonlinear and infinite-dimensional.~\cite{Kuramoto1984chemical} Its  solution shed analytical light on a phase transition to mutual synchronization that Winfree had  previously discovered in a similar but less convenient system of oscillators.~\cite{winfree1967biological, winfree1980} Since then, the  Kuramoto model has been an object of fascination for nonlinear dynamicists, as well as a  simplified model for many real-world instances of coupled oscillators in physics, biology,  chemistry, and engineering.~\cite{strogatz2000Kuramoto, pikovsky2003synchronization, strogatz2003sync, acebron2005Kuramoto, dorfler2014synchronization, pikovsky2015dynamics, rodrigues2016Kuramoto, bick2020understanding} 

From a mathematical standpoint, one of the most intriguing problems has been to  explain the tractability of the Kuramoto model. What symmetry or other hidden structure  accounts for its solvability?  

The first clues came from work on an adjacent topic: series arrays of $N$  identical overdamped Josephson junctions. The governing equations for those superconducting  oscillators are closely related to the equations of the Kuramoto model~\cite{wiesenfeld1996synchronization, wiesenfeld1998frequency}, and themselves displayed  remarkable dynamical features, such as surprisingly low-dimensional invariant tori\cite{tsang1991dynamics,swift1992averaging} and ubiquitous neutral stability of splay states\cite{nichols1992ubiquitous}, despite the presence of  damping and driving in the governing equations. These features were explained in 1993 by the  discovery of a certain change of variables, now called the Watanabe-Strogatz transformation~\cite{watanabe1993integrability,watanabe1994constants},  which showed that the governing equations have $N-3$ constants of motion for all $N \geq 3$. Goebel~\cite{goebel1995comment} then pointed out that the Watanabe-Strogatz transformation could be viewed as a time-dependent version of a linear fractional transformation, a standard tool in complex  analysis. For more than a decade, however, these results did not attract much attention,  perhaps because they were assumed to be restricted to problems about Josephson junctions,  and within that specialized setting, even further restricted to junctions that were strictly identical.  

A breakthrough occurred in 2008 with the work of Ott and Antonsen.~\cite{ott2008low,ott2009long} They found an  astonishing way to capture the macroscopic dynamics of the infinite-$N$ Kuramoto model, even  when the oscillators’ frequencies were non-identical and randomly distributed. First, they  wrote down an ansatz --- seemingly pulled out of thin air --- for the density function $\rho(\theta,  \omega, t)$ of oscillators having phase $\theta$ and intrinsic frequency $\omega$ at time $t$. Their ansatz had the form of a time-dependent Poisson density (a density better known for its  role in the study of partial differential equations, specifically for the solution of Laplace’s  equation on a disk, given the values of the unknown function on the bounding circle). By  making this ansatz of a Poisson density, Ott and Antonsen reduced the infinite-$N$ Kuramoto  model, an integro-partial differential equation, to an infinite set of coupled ordinary differential  equations. Then, by further assuming that the intrinsic frequencies of the oscillators were randomly distributed  according to a Lorentzian (Cauchy) distribution, Ott and Antonsen showed that the order parameter dynamics of the Kuramoto model could be reduced tremendously, all the way down  to an ordinary differential equation for a single scalar variable, the amplitude of the order  parameter.~\cite{ott2008low} With this discovery, the floodgates were now open. Almost immediately the Ott-Antonsen ansatz was used to solve many longstanding problems about the Kuramoto model  and its variants, as well as to generate and solve many new problems.~\cite{pikovsky2015dynamics}  

Still, a lot of old questions hung in the air. Both the Watanabe-Strogatz transformation  and the Ott-Antonsen ansatz appeared somewhat unmotivated and almost miraculous. Where  did they come from, and why did they work? It also was not clear whether they were  connected or perhaps even equivalent. There were reasons to doubt they were linked:  the Watanabe-Strogatz transformation could be used for any finite $N \geq 3$, but seemed restricted to identical oscillators, whereas the Ott-Antonsen ansatz allowed for non-identical  oscillators but seemed restricted to the continuum limit of infinite $N$. Also, why were linear fractional transformations and Poisson densities --- tools from other branches of  mathematics --- popping up in these studies of dynamical systems? 

Later work made sense of all of this. The Josephson arrays and the Kuramoto model  both turned out to have deep mathematical ties to group theory, hyperbolic geometry, and  projective geometry, and both the Watanabe-Strogatz transformation and the Ott-Antonsen  ansatz were tapping into these structures.~\cite{pikovsky2015dynamics,pikovsky2008partially,marvel2009identical,stewart2011phase,chen2017hyperbolic,chen2019dynamics} For the Josephson arrays, the governing equations  turned out to be generated by a group action, specifically the action of the Möbius group of linear fractional transformations of the unit disk to itself. Seen in this light, the constants of  motion for the Josephson arrays were cross-ratios, and the invariant tori were group orbits. The  same group-theoretic structure was found to underlie the Kuramoto model (in the special case  where all the oscillator frequencies are identical) as well as other sinusoidally coupled systems  of identical phase oscillators.\cite{marvel2009identical, chen2017hyperbolic}  

In the past few years, several researchers wondered how far this story could be pushed.  Are there quantum or higher-dimensional extensions of the Kuramoto model that might show  similar reducibility? A number of results along these lines have now been found.~\cite{lohe2009non,tanaka2014solvable,chi2014emergent, ha2016collective, ha2018relaxation, lohe2018higher, jacimovic2018low, chandra2019continuous, chandra2019complexity, lohe2019systems, deville2019synchronization, ha2021constants, jacimovic2021reversibility,dai2021d} In particular,  several researchers have explored a generalization of the Kuramoto model in which the  oscillators move on spheres instead of the unit circle. These spheres could be either the  ordinary two-dimensional sphere or higher-dimensional spheres. A counterpart of the Ott-Antonsen ansatz has been discovered for the continuum version of the Kuramoto model on the  $d$-dimensional sphere and used to reduce its infinite-dimensional dynamics to a lower dimensional set of ordinary differential equations (ODEs).~\cite{chandra2019complexity} But as before, some of the results appear disconnected and a bit miraculous.  

Our goal in this paper is to show that hyperbolic geometry and group theory can unify  and clarify our understanding of the Kuramoto model on a sphere and make all the latest results seem natural, just as they did before for the traditional Kuramoto model. Our approach  explains the model’s reducibility for any finite number of oscillators, as well as for the  continuum limit, and it reveals why Poisson densities arise again in this setting. There is a close  connection to Laplace’s equation and harmonic analysis, as we will see in Section V. We  also find that complex analysis is not really essential, which is just as well, since it does not  generalize to the higher-dimensional spheres being considered here. Instead, the proper  mathematical setting is harmonic analysis and hyperbolic geometry on higher-dimensional  balls. Our work also allows us to go beyond merely unifying existing results. For instance, by establishing that linearly coupled systems of identical Kuramoto oscillators on a sphere  have a hyperbolic gradient structure, we can prove new global stability results about convergence to the  synchronized state, as described in Section VIII.

\section{\label{sec:level1}Preliminaries}
\subsection{\label{sec:level2}The Kuramoto Model on a Sphere}
In a pioneering paper, Lohe~\cite{lohe2009non} observed that there are at least two natural generalizations of the Kuramoto model to higher dimensions. One of them replaces the phases $\theta_j$ of the original Kuramoto model\cite{Kuramoto1975self,Kuramoto1984chemical} with complex numbers $\exp(i \theta_j)$ on the unit circle and then views those as equivalent to $2 \times 2$ rotation matrices parametrized by a rotation angle $\theta_j$. From there, it is a natural step to consider other Lie groups of matrices, many of which are non-Abelian.

Our concern in this paper, however, is with a different  generalization of the Kuramoto model. Instead of regarding oscillators as particles moving on the unit circle, we think of them as particles moving on the unit sphere. The sphere could be the surface of the ordinary unit ball in three dimensions, or some higher-dimensional sphere $S^{d-1}$ in $\mathbb{R}^d$. When $d=2$, the sphere reduces to the unit circle in the plane, and the model reduces to the original Kuramoto model.

The governing equations for the \emph{Kuramoto model on a sphere} are 
\begin{equation}
 \dot x_i = A_ix_i + Z - \langle Z, x_i \rangle x_i, \quad i = 1, \dots, N,
 \label{governingeqn}
\end{equation}
where $x_i$ is a point on the unit sphere $S^{d-1} \subset {\Bbb R}^d$, each $A_i$ is an antisymmetric $d \times d$ matrix, and $Z \in {\Bbb R}^d$ is a $d$-dimensional vector analogous to the complex order parameter for the classic Kuramoto model. In Eq.~\eqref{governingeqn}, the matrix $A_i$ and the vector $Z$ are functions of the configuration $(x_1, \dots, x_N)$ of points on the sphere. Note that $Z$ does not depend on $i$; like the usual Kuramoto order parameter, it plays the role of a mean-field quantity that couples all the ``oscillators'' $x_i$ together. The antisymmetric matrix $A_i$ is the higher-dimensional counterpart of an intrinsic frequency $\omega_i$ in the original Kuramoto model. 

A straightforward computation shows that the dot product between an oscillator's instantaneous position and instantaneous velocity satisfies $\langle x_i, \dot{x_i} \rangle = 0$, which proves that oscillators that start on the unit sphere stay on it forever. The state space for this system is the $N$-fold product $X = (S^{d-1})^N$, which has dimension $N(d-1)$.  Later we will also consider the natural infinite-$N$ analogue of \eqref{governingeqn}, where a state is a probability measure on $S^{d-1}$. 

In what follows, we allow $Z$ to be any smooth function on the state space $X$, though in examples we usually restrict to fairly simple functions, like a linear combination of the form
$$
Z = \sum_{i=1}^N a_i x_i
$$
where the $a_i$ are real constants.  

\subsection{\label{sec:level2}General philosophy: Lie groups and reducible systems}

There is a general technique for dimensional reduction of systems like \eqref{governingeqn} which we pause to describe.  Suppose we have a smooth manifold $X$, which we think of as a state space,  and a group $G$ acting on $X$, where $G$ is also a smooth manifold (in other words, $G$ is a Lie group).  Then the group action induces a space of vector fields on $X$, the so-called  infinitesimal generators of the action.  

To construct these generators, let $\gamma(t)$ be a smooth curve in $G$ with $\gamma(0) = e$, the identity element in $G$. Then the derivative $\dot{\gamma}(0) = v$, where $v$ is a vector in the tangent space $T_eG$ of $G$ at $e$. This vector $v$ is in turn  associated very naturally with a corresponding vector $\Tilde{v}$  in the tangent space of $X$, as follows. For each $x \in X$, $t \mapsto \gamma(t)x$ is a smooth curve in $X$, and its derivative at $t = 0$ defines a vector $\Tilde{v}_x$ in the tangent space at $x$.  The vector field $\Tilde{v} = (\Tilde{v}_x)$ is the infinitesimal generator corresponding to the element $v$ in the tangent space of $G$ at $e$. At each point $x \in X$, the infinitesimal generators $v_x$ span a linear subspace $V_x$ of the tangent space $T_xX$, which is exactly the set of vectors tangent to the group orbit $Gx$ at the element $x$.  

Now suppose we have a vector field $\xi$ on $X$, which defines a dynamical system on the state space $X$.  If $\xi_x \in V_x$ at each point $x \in X$, then the flow corresponding to the vector field $\xi$ will be constrained to lie on the group orbits $Gx$.  If the dimension of $G$ is less than the state space $X$, then this will give us a dimensional reduction of the dynamics from $ \dim X$ to $\dim G$.  Now suppose, as is the case in applications of this methodology, the correspondence $G \to Gx$ is one-to-one for generic $x \in X$; equivalently, the stabilizer subgroup $G_x = \{e\}$ for generic $x \in X$.  Then for each $x \in X$, the flow on the group orbit $Gx$ is equivalent to a flow on $G$, which is a lower-dimensional space than the state space $X$.

Here is a familiar example: Consider the orthogonal group $G = SO(d)$ consisting of orientation-preserving linear isometries of ${\Bbb R}^d$.  (For an intuitive picture, think of these isometries as rotations.) Then $G$ acts on ${\Bbb R}^d$, and the corresponding infinitesimal generators are the linear vector fields $\nu_x = Ax$, where $A$ is any skew-symmetric matrix.  We can also let $G$ act on the product space $X = ({\Bbb R^d})^N$ of $N$-tuples $x = (x_i)$, $x_i \in {\Bbb R}^d$, and then the infinitesimal generators have the form $(\nu_x)_i = Ax_i$ for some skew-symmetric matrix $A$.  We could also, if we like, restrict $X$ to $N$-tuples $(x_i)$ with $x_i \in S^{d-1}$, the state space for \eqref{governingeqn}.  Now suppose we had a dynamical system on $X$ of the form
$\dot x_i = A x_i$, where the skew-symmetric matrix $A$ is a function of the configuration $x = (x_i)$; this is just the special case of \eqref{governingeqn} with $Z = 0$.  Then the dynamics on $X$ reduces to dynamics on $G$, which has dimension $d(d-1)/2$, and for large $N$ this is much smaller than the dimension of $X$, which is $N(d-1)$.  Basically the configuration $(x_i)$ of points on the sphere $S^{d-1}$ can only move collectively by a rotation of the sphere, so the dynamics reduces to a dynamical system on $SO(d)$.

We want to apply this methodology to the system \eqref{governingeqn}. But since that system generally has $Z \ne 0$, we need a different group action to make this strategy work. Fortunately, vector fields of the form seen in the Kuramoto model, 
\begin{equation}
 \dot x = Ax + Z - \langle Z, x \rangle x,
 \label{infinitesimalgs}
\end{equation}
 turn out to arise as the infinitesimal generators of the group action of a larger group $G$ acting on the sphere $S^{d-1}$ and its interior, the unit ball $B^d$.  This larger group is the M\"obius group of isometries of the hyperbolic geometry on $B^d$.  It contains the orthogonal group as a proper subgroup, but has bigger dimension $d(d+1)/2$.

So for the Kuramoto system \eqref{governingeqn}, if the matrices $A_i$ all happen to be identical, we can reduce the dynamics of the system to a much smaller system on this Möbius group $G$, and use this reduction to understand the dynamics on the larger state space $X$.  This is the essence of the approach we take. Ultimately we apply it to prove a synchronization theorem for the system \eqref{governingeqn} for the special order parameter 
$Z = \sum_{i=1}^N a_i x_i$ 
with $ a_i  > 0$. But first we need to show that vector fields on $S^{d-1}$ of the form in   \eqref{infinitesimalgs} are indeed the infinitesimal generators of the action of the larger M\"obius group.

\subsection{\label{sec:level2}Hyperbolic geometry and Möbius transformations}

In this paper, a \emph{Möbius transformation} is a composition of Euclidean isometries and spherical inversions of $\mathbb{R}^d$ mapping the unit ball homeomorphically to itself and preserving orientation. This is a more restrictive definition than the commonly defined Möbius transformations which in general do not need to preserve the unit ball. 

As in the case $d=2$, flows of the form (1) are intimately related to the natural hyperbolic geometry on the unit ball $B^d$ with boundary $S^{d-1}$.  This geometry has metric
$$
ds = {2 |dx| \over 1 - |x|^2},
$$
where $|dx|$ is the ordinary Euclidean metric. Isometries are assumed to be with respect to this hyperbolic geometry, unless otherwise qualified as Euclidean. The metric $ds$ has constant (sectional) curvature equal to $-1$, and we can describe its isometries, which generalize the M\"obius transformations preserving the unit disc for $d=2$. For $d = 2$, let $w \in B^2$ and consider the M\"obius transformation
$$
M_w(x) = {x-w \over 1-\overline w x},
$$
which preserves the unit disc $B^2$ and its boundary $S^1$.  To generalize this to higher dimensions, we need to express $M_w(x)$ without reference to complex arithmetic operations or conjugation.  This is the goal of the next subsection. 

\subsubsection{M\"obius transformations in higher dimensions}

Using the identity $2 \langle w, x \rangle = \overline w x + w \overline x$, we see that

\begin{align*}
{(x-w)  (1 - w \overline x) \over (1-\overline w x)(1 - w \overline x)} &= {x-w -w |x|^2 +w^2 \overline x\over 1 -2 \langle w, x \rangle +|w|^2|x|^2} \\ &= {x-w -w |x|^2 +w( 2 \langle w, x\rangle - \overline w x) \over 1 -2 \langle w, x \rangle +|w|^2|x|^2} \cr &= {(1-|w|^2) x-(1  - 2 \langle w, x \rangle +|x|^2) w \over 1 - 2\langle w, x\rangle + |w|^2 |x|^2}. 
\end{align*}
This form of $M_w$ generalizes to higher dimensions:  Let $w \in B^d$ and define
\begin{align*}
M_w(x) &=  {(1-|w|^2) x-(1  - 2 \langle w, x \rangle +|x|^2) w \over 1 - 2\langle w, x\rangle + |w|^2 |x|^2}  \\ &= {(1-|w|^2) (x-|x|^2 w)\over 1 - 2\langle w, x\rangle + |w|^2 |x|^2} -w,
\end{align*}
where $x \in B^d$ or $S^{d-1}$.  We call $M_w$ a \emph{boost transformation}.  

If $|x| = 1$ this formula simplifies to
$$
M_w(x) = {(1-|w|^2) (x-w) \over |x-w|^2 } -w.
$$
Now see that
\begin{align*}
M_w(w) &=  {(1-|w|^2) w-(1  - 2 \langle w, w \rangle +|w|^2) w \over 1 - 2\langle w, w\rangle + |w|^2 |w|^2} \\ &= {(1-|w|^2) w-(1  - |w|^2) w \over 1 - 2\langle w, w\rangle + |w|^2 |w|^2} = 0.
\end{align*}
Alternatively, we can use the second form to show
$$
M_w(w)    = {(1-|w|^2) (w-|w|^2 w)\over 1 - 2\langle w, w\rangle + |w|^2 |w|^2} -w = { (1-|w|^2)^2 w \over (1-|w|^2)^2 } - w = 0.
$$ 
Similar computations show that $M_0$ is the identity, $M_w^{-1} = M_{-w},$ and $M_w(0) = -w$. 

It is a standard result in hyperbolic geometry (e.g., see Theorem 3.5.1 in Beardon~\cite{ab83}) that any orientation-preserving isometry of $B^d$ can be expressed uniquely as the product of a boost and a rotation (an orientation-preserving orthogonal transformation), and these operations can be done in either order. In other words, any such  isometry can be written uniquely in the form
$$
g(x) = \zeta M_w(x)
$$
and also uniquely in the form
$$
g(x) = M_{-z} (\xi x),
$$
for some vectors $w, z \in B^d$ and rotations $\zeta, \, \xi \in SO(d)$,  where $SO(d)$ denotes the group of orientation-preserving orthogonal linear transformations on ${\Bbb R}^d$. So counting the extra $d$ dimensions that we get from the vector $w$ or $z$,  we see that the Möbius group has dimension $d + d(d-1)/2 = d(d+1)/2$.

Depending on the situation, one of these two forms might be more useful than the other, even though they are equivalent. In the interest of flexibility, it is useful to find how the parameter pairs $w,\zeta$ and  $z,\xi$ are related. We can do this by comparing the linearizations of the two formulas for $g(x)$ at $x=0$: 
\begin{align*}
    g(x) \approx \zeta  \left(-w + (1-|w|^2)x \right) \approx z + (1 - |z|^2)\xi x.
\end{align*}
Equating coefficients implies $z = - \zeta w$ (hence $|z| = |w|$) and $\xi = \zeta$.

\subsubsection{Infinitesimal generators}
Having parametrized the Möbius transformations, we are now ready to derive the associated infinitesimal generators of the Möbius group action on the ball $B^d$. We will show that they correspond to flows of the form 
\begin{equation}
\dot y = Ay -  \langle Z, y \rangle y + { 1\over 2}  (1+|y|^2) Z , \label{infinitesimalflow}
\end{equation}
where $A$ is an antisymmetric $d \times d$ matrix and $Z \in {\Bbb R}^d$ is a vector. Note that, as advertised, this flow extends to a flow on $S^{d-1}$ of the Kuramoto form in \eqref{infinitesimalgs}, as we can see by restricting \eqref{infinitesimalflow} to vectors $y$ on the unit sphere where $|y|=1$.

To derive \eqref{infinitesimalflow}, we work separately with the boost and rotation components. Let us start with the boost component. Replace $w$ by $tw$ and expand $M_{tw}(x)$ to first order in $t$: $$
M_{tw} (x) \approx {x - |x|^2 tw \over 1-2 t \langle w,x \rangle} -tw \approx x + t \left( 2 \langle w, x\rangle x  - (1+|x|^2)w \right).
$$
The derivative of this expression at $t=0$ (which is just the coefficient of $t$) gives us the infinitesimal generator: it  is an ``infinitesimal boost'' of the form \eqref{infinitesimalflow} with $Z = -2w$ and $A = 0$.  Next, recall that the infinitesimal generators corresponding to the rotation components are flows of the form $\dot x = Ax$ for an antisymmetric matrix $A$. Together with the infinitesimal boosts we then get all flows of the form \eqref{infinitesimalflow}.  The group $G$ acts on the space $X$ in the natural way (component by component) and the infinitesimal generators of this group action on $X$ are flows of the form (1) with all $A_i$ identical.  Therefore, by the general philosophy discussed earlier, the evolution of any initial point $p \in X$ under the system (1) with all $A_i = A$ lies in the group orbit $Gp$.

\section{\label{sec:level1}Reduced Equations}
The given Kuramoto system has $N (d-1)$ degrees of freedom, for some large $N$. However, since the flow of the system is determined via an action of the $d(d+1)/2$ dimensional Lie group $G$, we can alternatively study the auxiliary dynamical system on $G$, which we call the reduced equations. By ignoring rotations, we can further restrict our attention to a system on the $d$-dimensional quotient $G/SO(d) \cong B^d$. The dimensional reduction not only makes the reduced equations easier to analyze than the original Kuramoto system, but the reduced equations require fewer computational resources to numerically integrate.

Now suppose all the terms $A_i$ in (1) are equal.  Fix a base point $p = (p_1,  \dots, p_N) \in X$.  Then if the points  $p_i$ are in sufficiently general position, every element in the $G$-orbit of $p$ can be expressed uniquely as $g p$ for some $ g \in G$, with parameters $w, z$ and $\zeta$.  We wish to derive the corresponding evolution equations for $w, z$ and $\zeta$.  Let $(x_i(t))$ be any solution to (1) in the group orbit $Gp$; we do not require that the initial point $(x_i(0) )= p$.  Then for $i = 1, \dots, N$ we have $x_i(t) = g_t (p_i)$ for a unique $g_t \in G$, which determines the parameters $w, z, \zeta$ as functions of $t$.  Now consider the equation \eqref{infinitesimalflow}, with coefficients $A$ and $Z$ evaluated at  $(x_i(t))$.  This is a non-autonomous ODE on $\overline {B^d}$, and its time-$t$ flow must be given by some $\tilde g_t \in G$.  This ODE has solutions $x_i(t) = g_t (p_i) = g_t (g_0^{-1}( x_i(0)))$, which implies that $\tilde g_t = g_t g_0^{-1}$.

So for any $y_0 \in \overline {B^d}$,
$$
y(t) = g_t ( g_0^{-1} (g_0 (y_0)) ) = g_t (y_0) = \zeta M_w (y_0) = M_{-z}(\zeta y_0)
$$ 
must satisfy the ODE \eqref{infinitesimalflow}
with $A$ and $Z$ evaluated  at $(x_i(t))$ at time $t$.  In particular, if we let $y_0 = 0$, then $y(t) = -\zeta w = z$, so $z$ satisfies the ODE \eqref{infinitesimalflow}.

Now expand $y = \zeta M_w(y_0) = M_{-z} (\zeta y_0)$ to first order in $y_0$, using the variables $z$ and $\zeta$:
$$ 
y \approx z + (1-|z|^2) \zeta y_0,
$$
so
$$
\dot y \approx \dot z -2 \langle \dot z, z \rangle \zeta y_0 + (1-|z|^2) \dot \zeta y_0.
$$
On the other hand, \eqref{infinitesimalflow} gives

\begin{align*}
\dot y &= A y + {1 \over 2} \left( 1 + | y|^2 \right) Z - \langle Z, y \rangle y \cr
&\approx Az + {1 \over 2} (1+|z|^2) Z - \langle Z, z\rangle z \cr
&+(1-|z|^2) \Bigl(A \zeta y_0 + \langle z, \zeta y_0 \rangle Z -\langle Z, z \rangle \zeta y_0 - \langle Z, \zeta y_0 \rangle z \Bigr).
\end{align*}
Setting $y_0 = 0$ gives the $\dot z$ equation
\begin{equation}
\dot z = Az + {1 \over 2} (1+|z|^2) Z - \langle Z, z\rangle z \label{zdot}
\end{equation}
as expected, and since $\langle Az, z \rangle = 0$, this in turn implies that 
$$
\langle \dot z, z \rangle = {1 \over 2} (1-|z|^2) \langle Z, z \rangle.
$$
 Equating the $y_0$ terms, factoring out $1-|z|^2$ and canceling the common term $\langle Z, z \rangle \zeta y_0$ gives
$$
\dot \zeta y_0 = A \zeta x_0 + \langle z, \zeta y_0 \rangle Z - \langle Z, \zeta y_0 \rangle z.
$$
Together, the last two terms above define a special type of antisymmetric operator of $\zeta y_0$:  Given any $y_1, y_2 \in {\Bbb R}^d$, define the antisymmetric operator $\alpha$ as
$$
\alpha (y_1, y_2)y = \langle y_1, y) y_2 - \langle y_2, y) y_1;
$$
this operator has range = ${\rm span}(y_1, y_2)$ providing $y_1$ and $y_2$ are linearly independent; otherwise $\alpha(y_1,y_2) = 0$.  Then for all $y_0 \in {\Bbb R}^d$,
$$
\dot \zeta y_0= A \zeta y_0 + \alpha (z, Z) \zeta y_0
$$
and therefore
$$
\dot \zeta = (A + \alpha (z, Z)) \zeta.
$$

Differentiating $z  = -\zeta w$ gives
$$
Az +  {1 \over 2} (1+|z|^2) Z - \langle Z, z\rangle z = - \zeta \dot w - \dot \zeta w
$$
so 

\begin{align*}
\zeta \dot w &=  (A + \alpha (z, Z)) z - Az - {1 \over 2} (1+|z|^2) Z + \langle Z, z\rangle z \cr
&= Az +|z|^2 Z- \langle Z, z\rangle z -Az -  {1 \over 2} (1+|z|^2) Z + \langle Z, z\rangle z \cr
&= -{1 \over 2} (1-|z|^2) Z;
\end{align*}
hence
\begin{equation}
\dot w = -{1 \over 2} (1-|w|^2) \zeta^{-1} Z.
\label{wdot}
\end{equation}

Summing up, the evolution equations for the $(z, \zeta)$ coordinate system on $Gp$ are
\begin{subequations}
\begin{equation}
\dot z = Az + {1 \over 2} (1+|z|^2) Z - \langle Z, z\rangle z
\end{equation}
\begin{equation}
\dot \zeta =  (A + \alpha(z, Z)) \zeta,
\end{equation}
\label{zzeta}
\end{subequations}
with $A$ and $Z$ evaluated at $M_{-z} (\zeta p)$, and for the $(w, \zeta)$ coordinate system on $Gp$ are

\begin{subequations}
\begin{equation}
\dot w = -{1 \over 2} (1-|w|^2) \zeta ^{-1} Z
\end{equation}
\begin{equation}
\dot \zeta =  (A - \alpha(\zeta w, Z)) \zeta,
\end{equation}
\label{wzeta}
\end{subequations}
with $A$ and $Z$ evaluated at $\zeta M_w(p)$.  Note that these equations generalize the evolution equations for the parameters $w$ and $\zeta$ given in Chen et al.\cite{chen2017hyperbolic} for the classic case $d=2$.


\section{\label{sec:level1} Comparison of Z Versus W Coordinates}

The $\dot z$ equation~\eqref{zdot} is an extension of the system equation \eqref{governingeqn} on $S^{d-1}$. However, for finite $N$, the $\dot z$ equation does not uncouple from $\zeta$, since $Z$ is evaluated at $M_{-z}(\zeta p)$.  The exception to this is in the infinite-$N$ limit:  if the base point $p$ is now the uniform density on $S^{d-1}$, then $\zeta p = p$ (the uniform density is invariant under rotations) and the density $M_{-z}(p) $ is a hyperbolic Poisson density on $S^{d-1}$ whose centroid is a function of $z$.  In the case $d = 2$, this Poisson density has centroid $z$. Unfortunately this simple relationship is false for $d \ge 3$ (we will give more details on this in the next section).

The advantage of the $\dot w$ equation~\eqref{wdot} is that for an order parameter function of the form
$$
Z = \sum _{i=1}^N a_i x_i, 
$$
with $a_i \in \Bbb R$, $\zeta$ drops out of the $\dot w$ equation and we get the reduced equation
$$
\dot w =  -{1 \over 2} (1-|w|^2) Z(M_w(p)). 
$$
The parameter $w$ essentially defines the ``phase relations'' among the $x_i$; two configurations have the same $w$ if and only if they are related by a rotation.  So $w$ is the key parameter that determines whether the system is approaching synchrony or incoherence.

The $w$ variable also has a nice invariance under change of base points.  Suppose $p' = M(p) \in Gp$; then we have coordinates $w', \zeta'$ associated to the base point $p'$.  Any $q \in Gp$ has two expressions
$$
q  = \zeta M_w(p) = \zeta' M_{w'} p' = \zeta' M_{w'} (M(p)).
$$
Assuming the coordinates of $p$ are in sufficiently general position, this implies $\zeta M_w = \zeta' M_{w'} \circ M$, and hence
$$
0 = \zeta M_w(w) = \zeta' M_{w'} (M(w)).
$$
But the unique solution to $M_{w'}(y) = 0$ is $w'$, and hence $w' = M(w)$.  In other words, the coordinates $w$ and $w'$ transform exactly as the base points $p$ and $p'$. 


\section{\label{sec:level1}Continuum Limit}

Next, we consider the dynamics of the Kuramoto model \eqref{governingeqn} in the limit $N \to \infty$.  Let us assume first that the rotation terms $A_i = A$ are constant across the population, corresponding to identical ``oscillators.'' Later we will consider the case where $A$ varies depending on some distribution.  

Let us also assume that the order parameter $Z$ is is proportional to the centroid of the population: 
$$
Z = {K \over N} \sum_{i = 1}^N x_i
$$
In the continuum limit, a state of the system is a probability measure $\rho$ on $S^{d-1}$, and the order parameter becomes 
$$
Z = K \int_{S^{d-1}} x \, d \rho(x).
$$
The measure $\rho$ evolves according to the  continuity equation (also known as the noiseless Fokker-Planck equation) associated to the flow in (1).  Naturally, this flow must preserve group orbits under the action of $G$.  Recall that if $M \in G$, then the measure $M_\ast \rho$ is defined by the adjunction formula
$$
\int_{S^{d-1}} f (x) \, d (M_\ast \rho)  (x) = \int_{S^{d-1}} f(M (x))  \, d \rho (x) .
$$

In particular, we can consider the $G$-orbit of the uniform probability measure $\sigma$ on $S^{d-1}$. This orbit is special; whereas a typical group orbit $G \rho$ has dimension equal to the dimension of $G$, namely $d(d+1)/2$, the orbit $G \sigma$ has dimension only $d$.  This is because the stabilizer of $\sigma$ is $SO(d)$; any rotation fixes $\sigma$, whereas the boosts deform $\sigma$.  Hence the orbit $G \sigma$ has dimension $d$.  Any element in $G \sigma$ can be written as $(M_{-z})_\ast \sigma$, with $z \in B^d$.  The evolution equation for $z$ is \eqref{zdot}, with
\begin{equation}
Z(z) = K \int_{S^{d-1}} x \, d (M_{-z})_\ast \sigma(x) = K \int_{S^{d-1}} M_{-z} (x) \, d \sigma(x). 
\label{Zz}
\end{equation}
In the case $d =2$ with $x = \zeta \in S^1$, we have
$$
d\sigma(\zeta) = {1 \over 2 \pi i } {d \zeta \over \zeta},
$$
so the integral 
$$
Z(z)  = {K \over 2 \pi i} \int _{S^1} {\zeta+z \over 1+\overline z \zeta } \cdot {d\zeta \over \zeta} = K {\zeta+z \over 1+\overline z \zeta } \Biggr |_{\zeta = 0}  = K z
$$
by the Cauchy integral formula.  Therefore \eqref{zdot} simplifies to the equation
$$
\dot z = i \omega z + {K \over 2} (1-|z|^2) z
$$
when $d=2$.  Unfortunately, the formula $Z(z) = Kz$ is not correct for $d \ge 3$; though as we shall see later, this formula \emph{is} correct in higher dimensions for the complex hyperbolic model in even dimensions, which we discuss in the next section.  For $d=2$ the two geometries agree, which explains the coincidence for $d=2$. 

Any Riemannian manifold $X$ has a Laplace-Beltrami operator $\Delta$ associated to its metric; functions $f$ on $X$ satisfying the equation $\Delta f = 0$ are called \emph{harmonic}.  For functions on the ball $B^d$ with the hyperbolic metric, this operator is
$$
\Delta_{hyp} = (1-|x|^2)^2 \Delta _{euc} +2(d-2)(1-|x|^2) \sum_{i=1}^d x_i {\partial \over \partial x_i},
$$
where 
$$
\Delta_{euc} = \sum_{i=1}^d {\partial^2 \over \partial x_i^2}
$$
is the standard Laplace operator (see Stoll\cite{s}, Chapter 3).  We will call solutions to the equation $\Delta_{hyp}f = 0$ \emph{hyperbolic harmonic functions}; for $d=2$ these coincide with ordinary (Euclidean) harmonic functions.  We can consider the hyperbolic analogue of the classical Dirichlet problem: given a continuous function $f$ on $S^{d-1}$, extend $f$ to a hyperbolic harmonic function $\tilde f$ on $B^d$.  Assuming this problem has a unique solution, then for any rotation $\zeta \in SO(d)$ we must have $\widetilde {f \circ \zeta} = \tilde f \circ \zeta$, since rotations preserve the hyperbolic metric.  If we average $f \circ \zeta $ on $S^{d-1}$ over all rotations $\zeta \in SO(d)$ we get the constant function
$$
f_{ave}  = \int_{S^{d-1}} f(x) \, d \sigma(x)
$$
on $S^{d-1}$, and any constant is hyperbolic harmonic on $B^d$.  Therefore the average on $B^d$ of $\widetilde {f \circ \zeta}= \tilde f \circ \zeta$ over all $\zeta \in SO(d)$ must be the constant $f_{ave}$.  But  $\tilde f(\zeta (0)) = \tilde f(0)$ for all $\zeta$, so we must have
$$
\tilde f (0) =\int_{S^{d-1}} f(x) \, d \sigma(x).
$$
Now let $z \in B^d$; since $M_{-z}$ preserves the hyperbolic metric, we must have $\widetilde{f \circ M}_{-z}= \tilde f \circ M_{-z}$, which implies
\begin{align*}
\tilde f(z) &= \widetilde{f \circ M}_{-z} (0) 
\\&=  \int_{S^{d-1}} f(M_{-z}(x)) \, d \sigma(x) 
\\&= \int_{S^{d-1}} f(x) \, d ((M_{-z})_\ast \sigma) (x).
\end{align*}
As shown in Chapter 5 in Stoll~\cite{s}, the measure $(M_{-z})_\ast \sigma$ is given by the formula
$$
d((M_{-z})_\ast \sigma) (x) = P_{hyp} (z, x) \, d \sigma (x),
$$
with hyperbolic Poisson kernel function
\begin{equation}
P_{hyp} (z,x) = \left ( {1-|z|^2 \over |z-x|^2 } \right)^{d-1}. 
\label{hyperbolicPoisson}
\end{equation}
Thus the solution to the hyperbolic Dirichlet problem with boundary function $f$ on $S^{d-1}$ is given by the hyperbolic Poisson integral
$$
\tilde f(z) = \int_{S^{d-1}} P_{hyp}(z,x) f(x) \, d \sigma(x), \quad z \in B^d.
$$
The orbit $G \sigma$ consists of all \emph{hyperbolic}  Poisson measures $P(z,x) \, d \sigma (x)$, parametrized by $z \in B^d$.  By contrast, the Euclidean Poisson kernel function is
$$
P_{euc} (z,x) ={1-|z|^2 \over |z-x|^d },
$$
so the hyperbolic Poisson measures agree with the Euclidean Poisson measures only if $d=2$.

Now we can calculate the expression $Z(z)$ in the general case $d \ge 2$.  We see from \eqref{Zz} that $Z(z)$ is the hyperbolic Poisson integral of the function $Kx$ on $S^{d-1}$.  The function $Kx$ is (Euclidean) harmonic and homogeneous of degree 1 on ${\Bbb R}^d$; following the recipe in Chapter 5 in Stoll~\cite{s}, we see that its extension from $S^{d-1}$  to a hyperbolic harmonic function on $B^d$ is given by
\begin{equation}
Z(z) = K  {F(1, 1-d/2; 1+d/2; |z|^2) \over  F(1, 1-d/2; 1+d/2; 1) }z, 
\label{Zevaluated}
\end{equation}
where $F$ is the hypergeometric function
$$
F(a, b; c; t) = \sum_{k=0}^\infty { (a)_k (b)_k \over (c)_k } {t^k \over k!},
$$
with $(a)_0 = 1$ and $(a)_k = a(a+1) \cdots (a+k-1)$ for $k \ge 1$.  Notice that if $a$ or $b = 0$, then $F(a, b;c;t) = 1$; this gives $Z(z) = Kz$ for $d=2$, as expected.
\section{\label{sec:level1}Complex Case}
There is an alternative generalization of Kuramoto networks to higher-dimensional oscillators when $d = 2m $ is even.  Then ${\Bbb R}^d = {\Bbb C}^m$, and we can study systems of the form
\begin{equation}
 \dot x_j = A_jx + Z - \langle  x_j , Z \rangle x_j, \quad i = 1, \dots, N,
\label{complexflow}
\end{equation}
 where now $x_i $ is a point on the unit sphere $S^{2m-1} \subset {\Bbb C}^m$, $A_i$ is an anti-Hermitian $m \times m$ complex matrix, $Z \in {\Bbb C}^m$ and $\langle, \rangle$ denotes the complex-valued Hermitian inner product.  These systems are the same as the real case when $d=2, m=1$ but are different for $m \ge 2$.  To see this, suppose
$$
 Ax +Y - \langle x, Y\rangle_{\Bbb R} \, x = Bx + Z - \langle x, Z \rangle_{\Bbb C} \, x
 $$
for all $x \in S^{2m-1} \subset {\Bbb C}^m = {\Bbb R}^d$, where $A$ is antisymmetric, $B$ is anti-Hermitian, $Y, Z \in {\Bbb C}^m$ and we use the subscripts $\Bbb R$ and $\Bbb C$ to distinguish the real and complex inner products.  Then
 $$
(A-B)x = Z-Y + \Bigl(  \langle x, Y\rangle_{\Bbb R} -  \langle x, Z \rangle_{\Bbb C} \Bigr ) x
$$
and so $(A-B)(-x) = (A-B) x$ for all  $x \in S^{2m-1}$, which implies $A = B$. This implies
$$
Y-Z =  \Bigl (\langle x, Y\rangle_{\Bbb R} - \langle x, Z \rangle_{\Bbb C} \Bigr ) x
$$
for all $x \in S^{2m-1}$, hence $Y-Z \in {\rm span}_{\Bbb C} (x)$ for all $x \in {\Bbb C}^m$; if $m \ge 2$, this implies $Y = Z$.  But then we have
$$
\langle x, Y\rangle_{\Bbb R} = \langle x, Y \rangle_{\Bbb C}
$$
for all $x \in {\Bbb C}^m$, which can only hold if $Y = 0$.  Hence for $m \ge 2$, the only flows simultaneously of the form (1) and \eqref{complexflow} have $Z = 0$ and $A$ anti-Hermitian.

Flows of the form \eqref{complexflow} are related to the complex hyperbolic geometry on the complex unit ball $B^m$ with the Bergman metric (see Rudin\cite{r80}, Chapter 1).  The orientation-preserving isometries of this metric are generated by unitary transformations $\zeta \in U(m)$ and boost transformations of the form
\begin{align*}
 M_w(x) &= { \sqrt{1-|w|^2} \, x  +  \left( {   \langle x, w\rangle \over 1+ \sqrt{1-|w|^2} } - 1 \right ) w \over 1 - \langle x, w \rangle} \\&=  { x - w  +   {   \langle x, w\rangle  w- |w|^2 x \over 1+ \sqrt{1-|w|^2} } \over 1 - \langle x, w \rangle}.
 \end{align*}
 Notice that when $m = 1$, this reduces to the standard complex M\"obius map $M_w$.  As in the real case $M_0$ is the identity, $M_w^{-1} = M_{-w}$, $M_w(w) = 0$ and $M_w(0) = -w$.  Any orientation-preserving isometry of $B^d$ can be expressed uniquely in the form
$$
g(x) = \zeta M_w(x) = M_{-z} (\xi x),
$$
where $w, z \in B^m$ but now $\zeta, \, \xi \in U(m)$, the complex unitary group.  Linearizing at $x =0$ gives
\begin{align*}
g(x) &\approx \zeta \left ( -w -\langle x, w \rangle w +  \sqrt{1-|w|^2} \,  x +  {   \langle x, w\rangle w \over 1+ \sqrt{1-|w|^2} } \right )  \cr
&\approx  \zeta \left ( -w  +  \sqrt{1-|w|^2} \,  x -  {  \sqrt{1-|w|^2} \langle x, w\rangle w   \over 1+ \sqrt{1-|w|^2} } \right )  \cr
& \approx  z  + \sqrt{1-|z|^2} \, \xi x -  {   \sqrt{1-|z|^2} \langle  \xi x, z \rangle z \over 1 + \sqrt{1-|z|^2}}
\end{align*}
which implies $z = - \zeta w$ (hence $|z| = |w|$) and $\xi = \zeta$, as before.
\begin{figure*}
\includegraphics[scale=0.6]{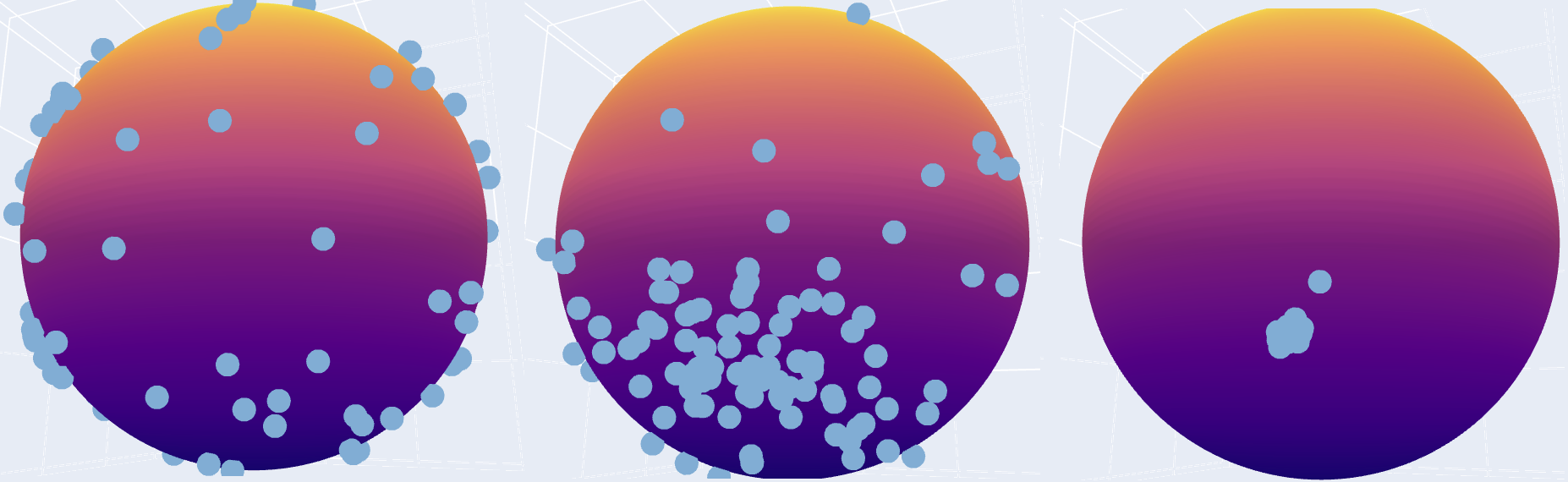}
\caption{\label{fig:epsart} A first-order linear Kuramoto system on the two-dimensional sphere $S^2$ with equal weights $a_i=1/N$, and randomly chosen initial conditions. The states shown are at $t = 0, t = 10,$ and $t = 40$ respectively. This simulation was written in Python and visualized with Plotly.}
\end{figure*}
The corresponding infinitesimal transformations are given by flows on the complex unit ball $B^m$ of the form
\begin{equation}
\dot y = Ay + Z -  \langle y, Z \rangle y, \label{yflow}
\end{equation}
with $A$ anti-Hermitian $m \times m$ and $Z \in {\Bbb C}^m$.  This flow extends to a flow on $S^{2m-1}$ of the form in \eqref{complexflow}.  Note the absence of the quadratic term $|y|^2 Z$ here. To derive these infinitesimal transformations, we can apply our prior power series expansion, noting that $|x|=1$ to obtain  
$$
M_{tw} (x) \approx {x -  tw \over 1- t \langle x, w \rangle} \approx x +  2\left(  \langle x, w \rangle x   -w \right)t.
$$
So the infinitesimal generator is an ``infinitesimal boost'' of the form  \eqref{yflow} with $Z = -\frac{1}{2}w$ and $A = 0$.  The infinitesimal generators corresponding to the rotation components are flows of the form $\dot x = Ax$ with $A$ anti-Hermitian; together with the infinitesimal boosts we get all flows of the form~\eqref{yflow}.  

\section{\label{sec:level1}Relation to Previous Research}
Many of the results above can be found in some form in the work of earlier authors.~\cite{lohe2009non,tanaka2014solvable,chi2014emergent, ha2018relaxation, lohe2018higher, jacimovic2018low, chandra2019continuous, chandra2019complexity, lohe2019systems, ha2021constants, jacimovic2021reversibility} Three papers in particular overlap considerably with the present work.

Tanaka~\cite{tanaka2014solvable} demonstrates in his 2014 paper that the dynamics of \eqref{governingeqn} can be reduced using Möbius transformations that fix the unit ball, similar to what Marvel, Mirollo, and Strogatz found\cite{marvel2009identical} for the traditional Kuramoto model. Tanaka writes his M\"obius transformations differently from ours, but he uses the same group of transformations and he also gets reduced equations for his M\"obius parameters. Tanaka's equation (10b) looks similar to our $\dot z$ equation~\eqref{zdot}, except without the $|z|^2$ term, which is puzzling. He does not mention the reduction down to dimension $d$ in the finite-$N$ case that we get with the $\dot w$ equation~\eqref{wdot}. Tanaka also notes that the complex case when $d = 2m$ is different, and generalizes the Ott-Antonsen residue calculation to this case, which is the highlight of his paper.  In the real case, Tanaka's equation (15) is similar to our equation \eqref{Zevaluated}, though we were not able to show that the two expressions are equivalent. Finally, Tanaka also presents a generalization of the Ott-Antonsen reduction\cite{ott2008low} for the complex version of the system. 

Lohe~\cite{lohe2018higher} also looks at the same system as (1) (see his equation (22)) and he derives a similar reduction as ours by using Möbius transformations for the finite-$N$ model. His transformation (30) on $S^{d-1}$ is our $M_w$ (with $v = w$) and his equation (31) is the same as our $\dot z$ equation~\eqref{zdot}.  He also has something that looks like the $\dot w$ equation~\eqref{wdot}, which he says is independent of the rest of the reduced system for (in our notation) an order parameter function of the form
$$
Z = {1 \over N} \sum_{i = 1}^N \lambda_i Q_ix_i,
$$
where $Q_i \in O(d)$ and $\lambda_i \in \Bbb R$.  But such a $Z$ does not satisfy the identity $ \zeta Z(p) = Z(\zeta p)$ for all rotations $\zeta$, unless $Q_i = \pm  I$, so we do not see how the $\zeta$ term cancels. 

Lohe's map $M$ in his equation (55) (ignoring the $R$ factor) agrees with our map $M_{-v}$ on the sphere $S^{d-1}$, but not on the ball $B^d$.  So it is not a M\"obius transformation of the type we are using.  For example, $M(-v) = v$ whereas $ M_{-v} (-v) = 0$.  We are not sure why Lohe~\cite{lohe2018higher} prefers these maps over the boosts; he claims that $M$ preserves cross-ratios, but we do not see why this is advantageous.  His map $F$ in equation (63) (again ignoring the $R$ factor) is exactly our $M_{-v}$.

Chandra, Girvan, and Ott~\cite{chandra2019complexity} concentrate on the infinite-$N$ or continuum limit system, and derive a dynamical reduction for a special class of probability densities on $S^{d-1}$, generalizing the Poisson densities used in the Ott-Antonsen reduction. They proceed directly to the infinite-$N$ version of (1). They make a very clever guess (their equation (7)) of the form of the special densities that generalize the Poisson densities for $d=2$, and then calculate the exponent in the denominator of their expression, getting exactly the hyperbolic Poisson kernel densities  in~\eqref{hyperbolicPoisson} above. Their equation (15) is exactly the same as our $\dot z$ equation~\eqref{zdot} in the infinite-$N$ limit.  The integral in their equation (19) can be evaluated, as shown above in~\eqref{Zevaluated}.

\section{\label{sec:level1}An Example:  First-Order Linear Order Parameter Gives Gradient System}

We conclude with an analysis of the system (1) with a weighted order parameter 
\begin{equation}
Z = \sum_{i=1}^N a_i x_i, \label{orderparameter}
\end{equation}
where the $a_i$ are real constants.  

\begin{figure*}
\includegraphics[scale=0.7]{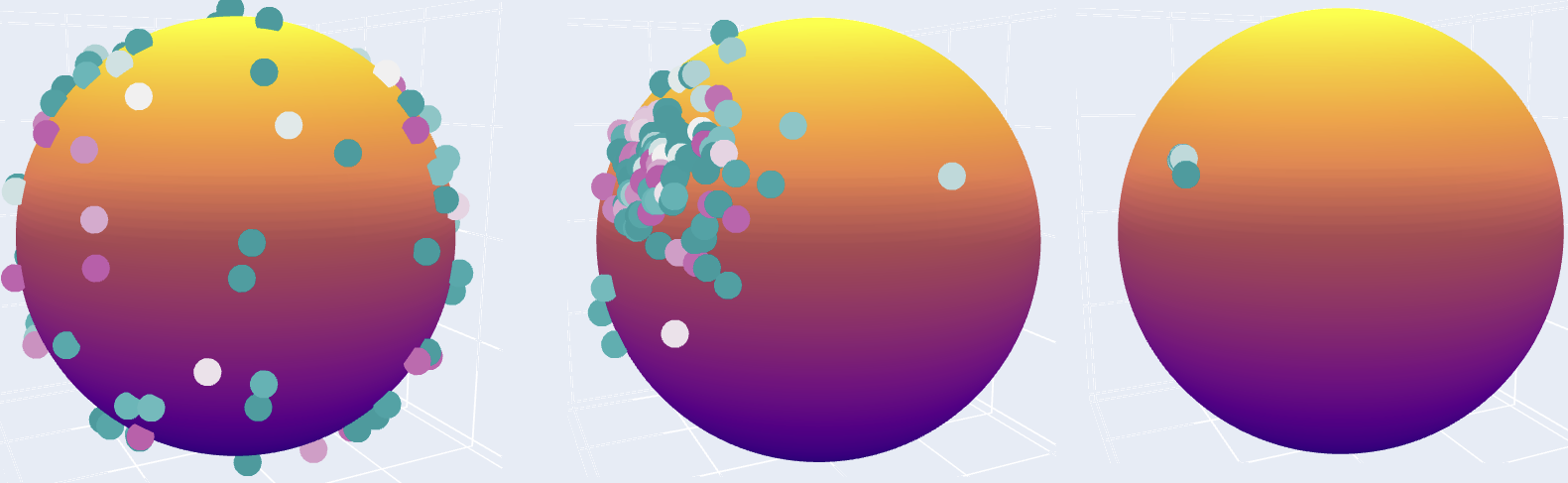}
\caption{\label{fig:epsart} A first-order linear Kuramoto system on $S^2$ with weights distributed according to a Riemann sum which approximates the integral of a normal distribution, and randomly chosen initial conditions. Pink particles contribute to the order parameter with greater weights than the blue particles do. The states shown are at $t = 0, t = 10,$ and $t = 40$ respectively.}
\end{figure*}
\begin{figure*}
\includegraphics[scale=0.7]{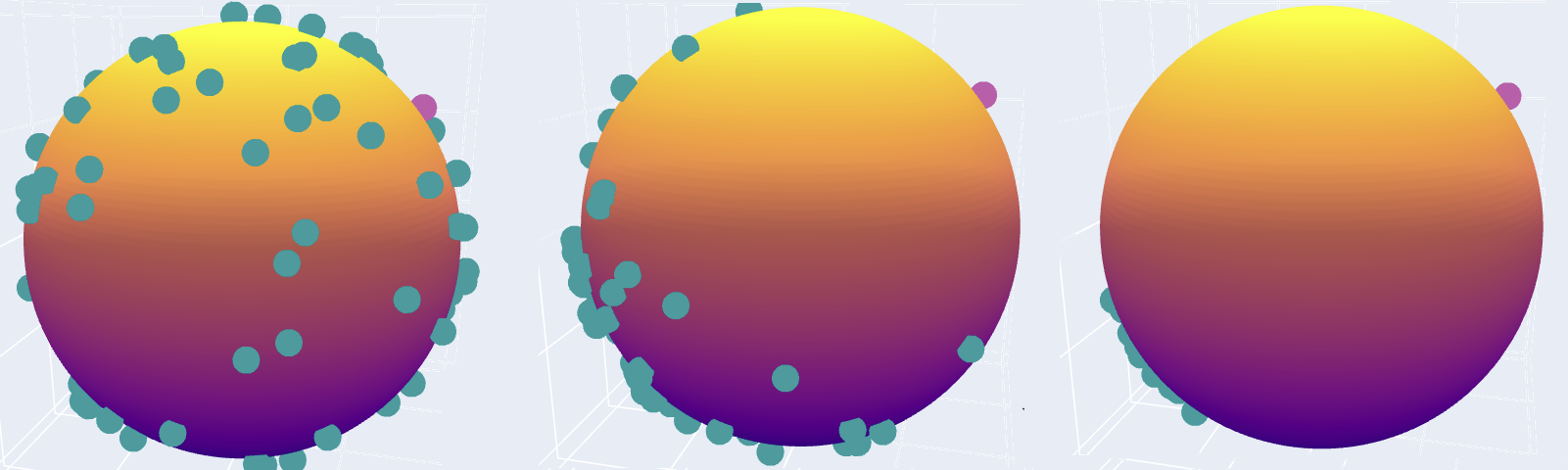}
\caption{\label{fig:epsart} A first-order linear Kuramoto system on $S^2$ with a majority cluster, where one particle is chosen to have a weight which exceeds the combined weights of all other particles (or equivalently, where all the particles have equal weight but a majority of them cluster into a single point and therefore act if they were a single giant particle; hence the name ``majority cluster''). The states shown are at $t = 0, t = -10,$ and $t = -40$ respectively; we have chosen to depict time running backward to highlight that the backwards-time limit tends toward an antipodal configuration. In this simulation, one particle, depicted in pink, was chosen to have a weight of $0.6$, and the remaining $99$ particles, depicted as blue, were chosen to have equal weights of $0.4/99$.}
\end{figure*}
\subsection{\label{sec:level2}A Computer Visualization}

Underlying all of the mathematics presented is a system of particles on a sphere which coalesce when certain conditions are met. To visualize this, we implemented the Runge-Kutta algorithm to numerically solve the Kuramoto model on the 2-dimensional sphere in 3-space, corresponding to the case $d=3$ in \eqref{governingeqn}.  For simplicity, we simulated $N = 100$ particles with equal weights ($a_i = 1/N$ in the order parameter $Z$) and  set the rotation term $A$ to zero for all the particles, a choice that is tantamount to ignoring the rotational influence, or equivalently, rotating the frame of reference along with the entire system as it evolves. 

In the simulation shown in Fig.~1, randomly chosen points on the sphere were used as initial conditions. As time increases, one can see that the particles coalesce to a limit point, mimicking the spontaneous synchronization that is well known for the traditional Kuramoto model ($d=2$) when the oscillators are identical.  

Later in this section, we will prove this synchronization behavior holds more generally for Kuramoto models on the sphere having weighted order parameters of the form \eqref{orderparameter}, provided the weights $a_i$ are all positive and satisfy an upper bound. Specifically, the forward-limit synchronization behavior remains regardless of how the particles are weighted, provided the $a_i$ sum to $1$ and no individual weight exceeds $1/2$, as we later prove. 

Figure 2 shows a simulation in which we weighted each particle according to the terms in a Riemann sum approximating the integral of the normal probability distribution. The pink particles, which have higher weights, exert greater influence over the final synchronization location of the particles, but there is still synchronization.

When time runs backwards, almost all initial conditions of the particles will tend towards a limiting configuration where their centroid is at the origin. The exception is when we have a majority cluster, depicted in Fig. 3, where one particle has a weight which exceeds the weight of all other particles. When this occurs, it is impossible to arrange the particles so their weighted centroid is at the origin, so the backwards time limit will tend towards an antipodal configuration, where all particles not in the majority cluster will coalesce around the antipode of the cluster. This is the configuration which minimizes the magnitude of the weighted centroid. We do not include the proof that the backwards-time limit is antipodal in this case, but it is a straightforward generalization of the result which we do prove below.

\subsection{\label{sec:level2}Existence of Hyperbolic Gradient}
As mentioned above, if $Z$ has the form in \eqref{orderparameter}, then the $\dot w$ equation in \eqref{wdot} reduces to
\begin{equation}
\dot w = -{1 \over 2} (1-|w|^2) Z(M_w(p)), \label{flow}
\end{equation}
independent of the parameter $\zeta$.  We will show that this is a gradient flow on the unit ball $B^d$ with respect to the hyperbolic metric.  In the presence of a Riemannian metric we can associate a $1$-form to any vector field, and the vector field is gradient if and only if the associated $1$-form is exact; since the unit ball is simply connected, this holds if and only if the associated 
$1$-form is closed.  For the Euclidean metric on $B^d$ (or any open subset of ${\Bbb R}^d$) and standard coordinates $w_1,  \dots, w_d$, the $1$-form associated to the vector field with components $f_1,  \dots, f_d$ is
$$
\omega = f_1 \, dw_1 + \cdots + f_d \, dw_d.
$$
If we scale the Euclidean metric by a positive smooth function $\phi$, then the associated $1$-form with respect to the metric $ds = \phi |dw|$  is then
$$
\omega = \phi^2 ( f_1 \, dw_1 + \cdots + f_d \, dw_d).
$$
Therefore the gradient of a function $\Phi$ with respect to this scaled metric is given by
$$
\nabla \Phi = \phi ^{-2} \nabla_{euc} \Phi,
$$
where $\nabla_{euc}$ denotes the ordinary Euclidean gradient operator. We have $\phi(w) = 2(1-|w|^2)^{-1}$ for the hyperbolic metric, so the hyperbolic gradient operator on $B^d$ is given by
$$
\nabla_{hyp} \Phi (w) = {1 \over 4} (1-|w|^2)^2 \nabla_{euc} \Phi (w).
$$

Now let's consider the vector field $V$ defined by \eqref{flow}.  By linearity, it suffices to treat the case $Z = x_i$, and we can take $i=1$ without loss of generality.   Then the associated $1$-form is

\begin{align*}
\omega &= {4  \over ( 1- |w|^2)^2 } \left( -{1 \over 2} (1-|w|^2)  \right) \\ &\cdot \sum_{j=1}^d  \left( {(1-|w|^2) (p_{1,\, j}  - w_j)  \over | p_1 - w|^2} - w_j \right) \, dw_j \cr \\
&= -2  \sum_{j=1}^d  \left(  {p_{1,\,  j}  - w_j  \over | p_1 - w|^2 } - { w_j  \over 1-|w|^2} \right) \, dw_j 
\end{align*}
where $p_{1, \, j}$ denotes the $j$th component of the point $p_1 \in S^{d-1}$.  Let  $E_j$ denote the coefficient of $dw_j$ in parentheses above; then
$$
d \omega = -2 \sum_{j, k = 1}^d {\partial E_j \over \partial w_k} \, dw_k \wedge dw_j.
$$
 Applying the chain and quotient rules gives
$$
{\partial E_j \over \partial w_k} = { 2(p_{1,\, j}  - w_j) (p_{1,\, k}  - w_k) \over |p_1-w|^4} + {2 w_j w_k \over (1-|w|^2)^2}
$$
for $j \ne k$, which is symmetric in $j$ and $k$; hence the sum above for $d \omega$ simplifies to $d \omega = 0$. Thus $\omega$  is closed and we see that the flow~\eqref{flow} is gradient for any order parameter function of the form~\eqref{orderparameter}.

Next, we show that the hyperbolic potential for $V$, up to an additive constant, is given by
\begin{equation}
\Phi(w)   = \sum\limits_{i = 1}^N a_i \log{1 - |w|^2 \over |w - p_i|^2} ={1 \over d-1} \sum\limits_{i = 1}^N a_i \log P_{hyp}(w, p_i).\label{potential}
\end{equation}
Here we follow the convention that the potential  \emph{decreases} along trajectories, so we are asserting  that $\nabla_{hyp} \Phi = -V$.   To derive this, use the identity $\nabla_{euc} |w-w_0|^2 = 2(w-w_0)$, for any constant vector $w_0 \in {\Bbb{R}}^d$.
Then
\begin{align*}
\nabla_{euc} \Phi (w) &=  \sum\limits_{i = 1}^N a_i \left( - {2w \over 1-|w|^2} - {2(w-p_i) \over |w-p_i|^2} \right) \cr
&=   {2 \over 1-|w|^2} \sum\limits_{i = 1}^N a_i \left(    {(1-|w|^2) (p_i-w) \over |w-p_i|^2} - w \right) \cr
&= {2 \over 1-|w|^2} \sum\limits_{i = 1}^N a_i M_w(p_i) = {2 \over 1-|w|^2} Z(M_w(p)).
\end{align*}
Hence we see that 
$$
\nabla_{hyp}  \Phi  (w) = {1 \over 2} (1 - |w|^2) Z(M_w(p)) = -V(w), $$
as desired. 

\subsection{\label{sec:level2}Analysis of Dynamics}
We can use the existence of the potential $\Phi(w)$ for the flow on $B^d$ to prove a global synchrony result for the system~\eqref{governingeqn} when the coefficients $a_i$ in the order parameter $Z$ are all positive.  Specifically, we assume that  $ 0 < a_i < 1/2$ for all $i$, and $\sum_{i = 1}^N a_i = 1$.  We also assume $N \ge 3$ and all the rotation terms $A_i$ in \eqref{governingeqn} are equal.  Under these conditions, almost all trajectories for (1) converge in forward time to the $(d-1)$-dimensional diagonal manifold $\Delta \subset X$ as $t \to \infty$, meaning that the system self-synchronizes.  In contrast, in backwards time the system tends to an incoherent state having zero order parameter:  as $t \to -\infty$ almost all trajectories for (1) converge to the codimension-$d$ subspace $\Sigma \subset X$ consisting of states with $Z(p) = 0$.

The proof is modeled after Theorem 1 in Chen et al.~\cite{chen2019dynamics} and will be based on two preliminary lemmas.  In each of these lemmas we assume the conditions on the $a_i$ above, and that the base point $p = (p_i)$ for the flow~\eqref{flow} has all distinct coordinates.  

We begin with a general observation about gradient flows in the ball $B^d$: if $w_0 \in B^d$ is any initial condition and $w^\ast \in B^d$ is in the forward limit set $\Omega_+(w_0)$, then $w^\ast$ is a fixed point for the flow.  To see this, let $\Phi$ be a potential for the flow, and suppose $w(t_n) \to w^\ast \in B^d$ for some sequence $t_n \to \infty$. 
Since the potential decreases along trajectories,
$$
\lim_{t \to \infty} \Phi(w(t)) = \lim_{n \to \infty} \Phi(w(t_n)) = \Phi(w^\ast).
$$
Let $F_t$ denote the time-$t$ flow map. If $w^\ast$ is not a fixed point, then for any $s > 0,$
\begin{align*}
\lim_{t \to \infty} \Phi(w(t)) &= \lim_{n \to \infty} \Phi(w(t_n+s)) \\
&= \lim_{n \to \infty} \Phi(F_s ( w(t_n))) 
\\&= \Phi(F_s(w^\ast)) \\ &< \Phi(w^\ast),
\end{align*}
which is a contradiction, so $w^\ast$ must be a fixed point.   (Compact limit sets are connected, so $\Omega_+(w_0)$ cannot consist of two or more but finitely many fixed points; however it is possible that forward or backward limits sets for gradient flows consist of a continuum of fixed points.  We will see that this is not the case for our system on $B^d$.)

\begin{lemma}
Any fixed point for the flow~\eqref{flow} in $B^d$ is repelling.
\end{lemma}

\begin{proof}
Suppose $w^\ast \in B^d$ is a fixed point for~\eqref{flow}.  As discussed above, an advantage of using the $w$-parameter is the equivariance with respect to change of base point $p$.  Consequently we can assume $w^\ast = 0$ without loss of generality, so $Z(p) =\sum\limits_{i = 1}^N a_i p_i =  0$.  To first order in $w$,
\begin{align*}
M_w(p_i ) &= { p_i - w \over 1 - 2 \langle w, p_i \rangle } - w \\&= (p_i - w) \bigl (1 + 2  \langle w, p_i \rangle \bigr ) -w \\&= p_i - 2w + 2  \langle w, p_i \rangle p_i.
\end{align*}
The linearization of~\eqref{flow} at the fixed point $w^\ast = 0$ is
\begin{align*}
\dot w &= -{1 \over 2} \sum_{i = 1}^N a_i \Bigl ( p_i - 2w + 2 \langle w, p_i \rangle p_i \Bigr )
\\&= w - \sum_{i = 1}^N a_i   \langle w, p_i \rangle p_i.
\end{align*}

We claim that the linear map
$$
Tw = \sum_{i = 1}^N a_i   \langle w, p_i \rangle p_i
$$
has $||T || <1$; to see this, suppose $|w| = 1$.  Then  $|  \langle w, p_i \rangle p_i | \le 1$ and $Tw$ is a convex combination of the vectors $\langle w, p_i \rangle p_i$.  We can only obtain $| Tw | = 1$  if all terms $\langle w, p_i \rangle p_i = u$ with $| u | = 1$, which implies all $p_i = \pm u$, and this cannot happen if at least three of the $p_i$ are distinct.
Hence $||T|| < 1$ and so the eigenvalues $\mu_i$ of $T$ satisfy $|\mu_i| < 1$.  The eigenvalues for the $\dot w$ linearization are $\lambda_i = 1 - \mu_i$, so wee see that ${\rm Re} \, \lambda_i > 0$ for all $i$, establishing that the fixed point $w^\ast$ is repelling.
\end{proof}

\begin{lemma}
$\displaystyle \lim_{|w| \to 1} \Phi(w) = -\infty.$
\end{lemma}
\begin{proof}
It suffices to show that
$$
\lim_{n \to \infty} \Phi(w_n) = -\infty
$$
for any sequence $w_n \in B^d$ with $w_n \to x \in S^{d-1}$.  The result is clear if $x \ne p_i$: as $n \to \infty$  the terms $|w_n - p_i|$ in the potential  \eqref{potential} are bounded away from $0$, and $1 - |w_n|^2 \to 0$.  So let's say that $w_n \to p_1$.  We rewrite $\Phi(w_n)$ as

\begin{align*}
\Phi(w_n) &= \log (1 - |w_n|^2) - 2 a_1 \log  |w_n - p_1| \\&-2  \sum _{i = 2}^N  a_i \log  |w_n - p_i| \\&= \log(1 - |w_n|) - 2 a_1 \log  |w_n - p_1| \\&+ \log (1 + |w_n|) - 2  \sum _{i = 2}^N  a_i \log  |w_n - p_i|.
\end{align*}
The latter two terms above have finite limit as $n \to \infty$, so we focus on the first two terms.  We have $ 1 - |w_n| \le |w_n-p_1| $, so
$$
 \log (1 - |w_n|) - 2 a_1 \log  |w_n - p_1|  \le (1 -2a_1) \log  |w_n - p_1|  \to -\infty
 $$
as $n \to \infty$, which proves our result.  Notice that we need the assumption $a_i < 1/2$ for this argument.
\end{proof}

\noindent \textbf{Theorem.} \emph{Under the conditions above, almost all trajectories for \eqref{governingeqn} converge to $\Delta$ as $t \to \infty$ and to $\Sigma$ as $t \to -\infty$.}
\begin{proof}
Let $p = (p_1, \dots, p_N) \in X$ be any point with all distinct coordinates.   The points on $Gp$ are parametrized by $w \in B^d$ and $\zeta \in SO(d)$, and the dynamics for these parameters are given by \eqref{wdot}. 
We begin with the dynamics as $t \to -\infty$.  Let $w(t)$ be a trajectory for~\eqref{flow} with initial condition $w_0 \in B^d$, and consider the backward time limit set $\Omega_-(w_0)$; this is a nonempty, compact, connected subset of $\overline{B^d}$.  The potential $\Phi$ is decreasing along all trajectories $w(t)$, hence bounded below as $t \to - \infty$, so Lemma 2 implies that the limit set $\Omega_-(w_0)$ must be contained in the interior $B^d$.  We know that any  $w^\ast \in \Omega_-(w_0)$ is a fixed point for the flow.    By Lemma 1, $w^\ast$ is repelling and so any trajectory $w(t)$ which comes sufficiently close to $w^\ast$ must have $w(t) \to w^\ast$ as $t \to -\infty$; therefore $\Omega_-(w_0) = \{w^\ast \}$.  This proves the existence of fixed points for~\eqref{flow}, and that every trajectory $w(t)$ converges to a fixed point as $t \to -\infty$.  If the flow had multiple fixed points, we would obtain a partition of $B^d$ into the disjoint open basins of repulsion of the fixed points, violating connectedness of the ball.  Therefore~\eqref{flow} has a unique fixed point $w^\ast$, and $w(t) \to w^\ast$ as $t \to -\infty$ for all trajectories.  The fixed point $w^\ast$ has $Z(M_{w^\ast}(p)) = 0$, so all trajectories in $Gp$ converge to $\Sigma$ as $t \to -\infty$.

In forward time, the limit set $\Omega_+(w_0)$ for any $w_0 \ne w^\ast$ must be completely contained in the boundary $S^{d-1}$, since the unique fixed point $w^\ast \in B^d$ is repelling.
Suppose we remove the factor $(1/2) (1 - |w|^2)$ in the flow~\eqref{flow}; the scaled vector field on $B^d$ given by
\begin{equation}
\dot w = -\sum_{i = 1}^N a_i M_w(p_i) = w -\sum_{i = 1}^N a_i  \left ( {(1-|w|^2) (p_i - w) \over |p_i - w |^2 }  \right ) \label{scaledsystem}
\end{equation}
has the same trajectories as the original flow, just with different time parametrizations.  Observe that this scaled vector field extends smoothly to ${\Bbb R}^d - \{p_i\}$, and coincides with the radial vector field $x$ at any $x \in S^{d-1}$ with $x \ne p_i$.  Therefore there is a unique trajectory passing through each point $x \in S^{d-1}$, flowing from the interior to the exterior of the sphere, as long as $x \ne p_i$.  Consequently the original flow~\eqref{flow} has a unique trajectory $w(t)$ in $B^d$ with $w(t) \to x$ as $t \to \infty$, as long as $x \ne p_i$.  This also shows that there is a neighborhood $U$ of $S^{d-1} -\{p_i\}$ such that if $w(t_0) \in U$ for some $t_0$, then $w(t) \to x \ne p_i$ for some $x \in S^{d-1}$.  So if $\Omega_+(w_0)$ contains some $x \ne p_i$, then the trajectory $w(t)$ of $w_0$ must enter the neighborhood $U$, and therefore $w(t) \to x \in S^{d-1}$ as $t \to \infty$.

Since limit sets are connected, the only other possibility is $\Omega_+(w_0) = \{ p_i \}$ for some $i$; equivalently, $w(t) \to p_i$.  We will show that there is a unique trajectory with this behavior for each $p_i$.  Assuming this, we see that with $N+1$ exceptions, any trajectory $w(t)$ converges to a point $x \in S^{d-1}$ with $x \ne p_i$ (the exceptions are the $N$ trajectories converging to the base point coordinates $p_i$, and the fixed point trajectory $w^\ast$).  The corresponding trajectory in $Gp$ has coordinates
$$
\zeta(t) M_{w(t)} (p_i) = \zeta(t) \left ( {(1-|w(t)|^2) (p_i - w(t)) \over |p_i - w(t)|^2 } - w(t) \right ).
$$
We have $|w(t)| \to 1$ and $ |p_i - w(t)|$ is bounded away from $0$ as $t \to \infty$, so $M_{w(t)} (p_i)  \to -x$ for each $i$ and therefore the trajectory $\zeta(t) M_{w(t)} (p)$ in $Gp$ converges to $\Delta$ as $t \to \infty$. 

This analysis breaks down at $x = p_i$ because the scaled vector field above does not have a unique limit as $w \to p_i$; rather, its limit depends on the direction of the approach.  To see this, write $w = p_1 -r u$, where $0 < r < 1$ and $| u | = 1$ (with this convention, $u = p_1$ corresponds to $w$ approaching $p_1$ radially).  Then $|p_1 - w| = r$ and 
$$
 |w |^2 = 1 - 2r \langle p_1, u \rangle + r^2
$$
so
$$
 {(1-|w|^2) (p_1 - w) \over |p_1 - w |^2 } = { (2r \langle p_1, u \rangle - r^2 ) ru \over r^2} =  \left ( 2\langle p_1, u \rangle - r \right) u.
 $$
As $r \to 0$, the magnitude of this term is $2 \langle p_1, u \rangle$, which depends on the angle of approach given by $u$
(note that $ \langle p_1, u \rangle > 0 $ because $u$ points outwards at $p_1$).

To complete the proof, we will examine the scaled system \eqref{scaledsystem} using the polar representation $(r, u)$, and show that the polar system has the unique fixed point $r^\ast = 0, u^\ast = p_1$, which has a unique attracting trajectory because it is a saddle with a $(d-1)$-dimensional unstable manifold.

We see that the scaled system has
\begin{align*}
\dot w &= p_1-ru  - a_1  \left ( 2\langle p_1, u \rangle - r \right) u +O(r)
\\&= p_1 - 2 a_1 \langle p_1, u \rangle u +O(r),
\end{align*}
where the $O(r)$ term is a smooth function of $r$ and $u$ for $|r|  < \epsilon = \min  |p_i-p_1|$, $i \ge 2$.  This condition insures that $|p_i - w| \ge |p_i - p_1| -|r| >0$, so the $i \ge 2$ terms in the scaled $\dot w$ equation are all smooth functions of $r$ and $u$.  And we can allow $r < 0$ here, even though it is not relevant to the $\dot w$ system.
Now $r^2 = |w-p_1|^2$, so
$$
r \dot r = \langle w-p_1, \dot w \rangle = -r \langle u, \dot w \rangle,
$$
which gives
$$
 \dot r  = -(1-2a_1) \langle p_1, u \rangle  +O(r).
$$
Differentiating $ru =  p_1 - w$ gives
\begin{align*}
r \dot u &=-\dot r u - \dot w \\ &= (1-2a_1) \langle p_1, u \rangle   u -  \Bigl ( p_1 - 2a_1\langle p_1, u \rangle u \Bigr ) + O(r) 
\\&= \langle p_1, u \rangle   u - p_1 +O(r).
\end{align*}
Hence the scaled system in polar form can be written

\begin{align*}
r \dot r &=  -(1-2a_1)   r  \,\langle p_1, u \rangle  +O(r^2), \\
r \dot u &=  \langle p_1, u \rangle \,  u -   p_1 +O(r).
\end{align*}
We emphasize that the $O(r)$ and $O(r^2)$ terms are smooth functions of $r, u$ as long as $|r| < \epsilon$.  We consider the ``semi-scaled'' polar system
\begin{subequations}
\begin{equation}
\dot r =  -(1-2a_1)   r  \,\langle p_1, u \rangle  +O(r^2),
\end{equation}
\begin{equation}
\dot u =  \langle p_1, u \rangle \,  u -   p_1 +O(r), 
\end{equation}
\label{semiscaled}
\end{subequations}
which has the same trajectories as the original system, just with different time parametrizations.  The advantage of this modified system is that the equations are smooth on $(-\epsilon, \epsilon) \times S^{d-1}$.  

Observe that the system \eqref{semiscaled} has $\{0\} \times S^{d-1} $ invariant, and has  fixed point $(r^\ast, u^\ast) = (0, p_1)$.  The fixed point $p_1$ is repelling on the invariant manifold $\{0\} \times S^{d-1} $; to see this, observe that
$$
 \langle p_1, u \rangle  \, \dot {} = \langle p_1, u \rangle ^2 - 1
$$
when $r = 0$.   In fact, if  we assign the coordinate $\theta$ on any great circle joining $p_1$ and $-p_1$ on $S^{d-1}$ so that $u = e^{i \theta}$ and $p_1 = 1$, then the system reduces to $\dot \theta = \sin \theta$.  We also see that the $\dot r$ equation linearized at $(0, p_1)$ is $\dot r =  -(1-2a_1)   r$, so the linearization of \eqref{semiscaled}  has the single negative eigenvalue $ -(1-2a_1)$ and $d-1$ positive eigenvalues $+1$.   Therefore  $(0, p_1)$ is a saddle with a one-dimensional stable manifold, and hence has a unique trajectory $(r(t), u(t)) \to (0, p_1)$ with $r(t) > 0$.
 
 Now suppose we have a trajectory $w(t) \to p_1$ in our original system~\eqref{flow}.  The corresponding trajectory for \eqref{semiscaled} will have $r(t) \to 0$; we cannot achieve $r(t) = 0$ in finite time because the manifold $\{ r = 0\}$ is invariant for \eqref{semiscaled}.  We must prove that $u(t) \to p_1$, so that this trajectory is in fact the saddle stable manifold.  Observe that
 $$
 \langle p_1, u \rangle  \, \dot {} = \langle p_1, u \rangle ^2 - 1 +O(r).
$$
Also note that $\langle p_1, u (t) \rangle > 0$ since $|w(t)| < 1$.   Let $0 < c < 1$; then for some $T \ge 0$, $t \ge T$ implies  $O(r(t)) \le ( 1-c^2)/2$.  Now suppose $0 < \langle p_1, u (t_0) \rangle < c$ for some $t_0 \ge T$; then $0 < \langle p_1, u (t) \rangle  c$ for all $t \ge t_0$. This is because the function $ t \mapsto \langle p_1, u (t) \rangle$ is decreasing if  $0 < \langle p_1, u (t) \rangle <  c$:
$$
 \langle p_1, u (t)  \rangle  \, \dot {} \le c^2 -1 +{1 \over 2} (1-c^2) = -\frac{1}{2}(1-c^2)
 $$
as long as $0 < \langle p_1, u (t) \rangle <  c$.  But this also implies that eventually $ \langle p_1, u (t)  \rangle  < 0$, which is a contradiction.  Hence we must have $  \langle p_1, u (t)  \rangle  \ge c$ for all $t \ge T$, which proves that $u(t) \to p_1$.
\end{proof}

\section{Summary and Discussion}

The natural hyperbolic geometry on the unit ball, with isometries consisting of the higher-dimensional M\"obius group, is key to understanding the dynamics of the Kuramoto model on a sphere.  Using this framework, we see that dynamical trajectories of \eqref{governingeqn} are constrained to lie on group orbits, and we can explicitly give the equations for the reduced dynamics on the group orbits.  For the special class of linear order parameters, the dynamics can be further reduced to a flow on the unit ball $B^d$, which is gradient with respect to the hyperbolic metric.  We analyze this flow and prove a global synchronization result for the system \eqref{governingeqn} for linear order parameters with positive weights and no weight greater than half the total.  This illustrates the power of the geometric / group-theoretic approach.

We conclude with some directions for future research.  The case of linear order parameters with both positive and negative weights can in principle be explored using similar methods; it will also reduce to a hyperbolic gradient system on the ball $B^d$.  In particular, the case when the sum of the weights is $0$ should have some intriguing dynamics. For the original ($d=2$) Kuramoto model, the dynamics are Hamiltonian and equivalent to the vector field on the unit disc given by placing a collection of point charges on the unit circle, with total charge $0$, as shown by Chen et al.~\cite{chen2019dynamics}.  Of course this result cannot generalize to all higher dimensions $d$; Hamiltonian dynamics is only possible in even dimensions.  

Another possible direction to explore is the case of systems where the oscillator population is divided into two or more families with different intrinsic rotational terms $A_i$, and the coupling across families differs from the coupling within families.  For the case $d = 2$, these systems often support ``chimera states,'' in which one or more  families synchronize while others tend to a partially disordered configuration.  These states can be dynamically stable within their M\"obius group orbits.  Our framework enables the dynamical reduction of multi-family networks of higher-dimensional oscillators, and makes possible the study of the dynamics of these networks without necessarily passing to the continuum limit, as is often done as a simplifying step in the analysis of Kuramoto networks.

\begin{acknowledgments}
Research supported in part by the NSF Research Training Group Grant: Dynamics, Probability, and PDEs in Pure and Applied Mathematics, DMS-1645643, and by NSF grant DMS-1910303. We thank Vladimir Jaćimović and Max Lohe for helpful comments on a preprint of this paper.
\end{acknowledgments}

\section*{Data Availability}
The data that supports the findings of this study are all available within the article.

\providecommand{\noopsort}[1]{}\providecommand{\singleletter}[1]{#1}%

\end{document}